\documentclass[10pt]{amsart}
\usepackage{verbatim}
\usepackage{amsmath}
\newtheorem{theorem}{Theorem}[section]
\newtheorem{proposition}{Proposition}[section]
\newtheorem{lemma}{Lemma}[section]

\theoremstyle{definition}
{}

\theoremstyle{remark}
\newtheorem{remark}{Remark}[section]

\newcommand{\spbp}{\sqrt{-1}\partial \bar{\partial}}

\newcommand{\pbp}{\partial \bar{\partial}}

\numberwithin{equation}{section}
\usepackage{hyperref}
\hypersetup{colorlinks,citecolor=blue,plainpages=false,hypertexnames=false}
\author[V. P. Pingali]{Vamsi Pritham Pingali}
\address{Department of Mathematics, Indian Institute of Science, Bangalore, India - 560012}
\email{vamsipingali@iisc.ac.in}
\begin{document}
\title{Non-abelian symmetric critical gravitating vortices on a sphere}
\begin{abstract}
We produce examples of solutions to the non-abelian gravitating vortex equations, which are a dimensional reduction of the K\"aher-Yang-Mills-Higgs equations. These are equations for a K\"ahler metric and a metric on a vector bundle. We consider a symmetric situation on a sphere with a relationship between the parameters involved (criticality), and perform a non-trivial reduction of the problem to a system of ordinary differential equations on the real line with complicated boundary conditions at infinity. This system involves a parameter whose dependence on the volume of the K\"ahler metric is non-explicit. We prove existence for this system using the method of continuity. We then prove that the parameter can be varied to make sure that all possible admissible volumes are attained.
\end{abstract}
\maketitle
\section{Introduction}\label{sec:intro}
\indent Cosmic strings are (as of now) hypothetical ``topological defects" in the Higgs field, that presumably appeared in the early universe due to spontaneous symmetry breaking \cite{Kibble}. The interaction of a certain kind of a cosmic string with gravity is modelled by the Einstein-Bogomol'nyi equations (in the special case of self-dual solutions of the more general Einstein-Maxwell-Higgs equations in the Bogomol'nyi phase). Typically these strings are assumed to have a ``cross-section" that is the entire plane $\mathbb{R}^2$. However, a more localised model is one that involves a large-radius sphere. A radially symmetric version of such strings was studied by Yang in \cite{Yang3} by reducing the equations to a \emph{single} ordinary differential equation (ODE), and using the shooting method to solve it. Yang produced solutions for \emph{some} volumes. In \cite{GaPiYa}, a vastly more general result was proven, which in particular implies that cosmic strings (with a spherical cross section) exist for any admissible volume. \\
\indent In \cite{AlGaGa2}, a more general system (than the Einstein-Bogomol'nyi system) called the gravitating vortex equations was proposed as a dimensional reduction of the K\"ahler-Yang-Mills (KYM) equations. The KYM equations themselves were introduced in \cite{AlGaGa} to analytically study the moduli space of triples $(X,L,E)$ of a compact polarised manifold $(X,L)$ equipped with a holomorphic vector bundle $E$. These equations admitted a moment map interpretation. This interpretation provided a Futaki-type invariant, whose vanishing was a necessary condition for existence \cite{AlGaGaPi, AlGaGaPiYa}. The vanishing of this invariant immediately led to the non-existence of certain cosmic strings (thus addressing a conjecture of Yang \cite{Yang3}). In \cite{AlGaGaPiYa2}, the necessity of a stability condition, and a uniqueness result were proven for the gravitating vortex equations using a symplectic reduction by stages approach. \\
\indent Motivated by the success of the KYM equations in the study of cosmic strings, and their connection to moduli spaces, the K\"ahler-Yang-Mills-Higgs equations were introduced in \cite{AlGaGa3}. Their dimensional reduction (using vortex vector bundles as opposed to vortex line bundles for the KYM equations) led to the non-abelian gravitating vortex equations. It is not yet known but is reasonable to speculate that these equations have connections to non-abelian cosmic strings arising from the electroweak theory (which is perhaps a more realistic approach to cosmic strings in the early universe) \cite{Yangbook}. In this paper, we produce examples of solutions to these non-abelian gravitating vortex equations on the sphere. Our method is to reduce to a \emph{system} of ODE. Unfortunately, the shooting method is rather formidable affair for systems. Therefore, we introduce a novel approach of using an Aubin-Yau style continuity method. This method might be useful in other contexts involving such kinds of ODE.\\
\indent Let $N>0, 0<l<N$ be integers. Consider $E=\mathcal{O}(N)\oplus\mathcal{O}(N)$ and $\phi= (\phi_1=z^l, \phi_2=z^{N-l})$ where $\phi_1, \phi_2 \in H^0(\mathbb{P}^1, \mathcal{O}(N))$. Using a permutation matrix, we may assume that $l\leq N-l$ without loss of generality. Let $\alpha, \tau>0$ be given parameters (the coupling constant and the symmetry-breaking parameter respectively). Assume, what we call, the criticality condition
$$2N\alpha \tau=1,$$
 and that the volume $2\pi V$ of a K\"ahler class $[\omega]$ satisfies the stability conditions $N<\frac{\tau V}{2} < 2N-2l$. We call such a volume, an \emph{admissible volume}.\\
 We want to solve the following system of equations for a K\"ahler metric $\omega \in [\omega]$ and a Hermitian metric $H$ on $E$.
\begin{gather}
iF_H+\frac{1}{2} \phi \otimes \phi^{*_H} \omega =\frac{\tau}{2} I\otimes\omega, \nonumber \\
Ricci(\omega)+\alpha(-2\spbp +\tau \omega)(\Vert \phi \Vert_H^2-2\tau)=0.
\label{eq:nonabelianvorticesactual}
\end{gather}
For the remainder of this paper, we shall omit the $I$ in $I\otimes \omega$, whenever it is clear from the context. The main result of this paper is as follows.
\begin{theorem}
There exists a smooth solution to \ref{eq:nonabelianvorticesactual} for any given admissible volume $V$.
  \label{thm:main}
\end{theorem}
To prove Theorem \ref{thm:main}, we use the following method of continuity in $s\in [0,1]$.
\begin{gather}
iF_H+\frac{1}{2} \phi \otimes \phi^{*_H} \omega =\frac{\tau}{2} \omega, \nonumber \\
Ricci(\omega)+\alpha(-2\spbp +\tau \omega)(\Vert \phi \Vert_H^2-2\tau)=(1-s)(Ricci(V\omega_{FS})+\alpha(-2\spbp +\tau V\omega_{FS})(\Vert \phi \Vert_{H_{s=0}}^2-2\tau)),
\label{eq:nonabelianvortices}
\end{gather}
where $\omega_{FS}$ is the Fubini-Study metric $\omega_{FS}=\frac{idz\wedge d\bar{z}}{(1+\vert z \vert^2)^2}=iF_{H_{FS}}$ (for some metric $H_{FS}$ on $\mathcal{O}(1)$). When $s=0$, the non-abelian vortex equation
\begin{gather}
iF_H+\frac{1}{2} \phi \otimes \phi^{*_H} \omega =\frac{\tau}{2} \omega,
\end{gather}
has a solution with $\omega=V\omega_{FS}$ thanks to the stability condition and Theorem 2.1.6 in \cite{Bradlow}. This metric combined with $\omega=V\omega_{FS}$ solves the second equation. However, we need a solution with a $U(1)$ symmetry (described in Remark \ref{rem:symmetries}). The existence of such a solution is proven in Section \ref{sec:symmetriesofvortex}.\\
\indent We shall prove that the set of $s$ for which solutions exist, is open (Section \ref{sec:continuityopenness}) and closed (Section \ref{sec:continuityclosedness}), hence proving existence at $s=1$. To this end, inspired by Yang \cite{Yang3}, we reduce this system to a partial differential equations (PDE) to a system of ODE on $\mathbb{R}$ with complicated boundary conditions at $\pm \infty$ (Section \ref{sec:ODE}) by imposing an $S^1$ symmetry. The trade-off with this method is that we introduce an arbitrary parameter $\lambda>0$ whose dependence on the volume parameter $V$ is complicated and non-explicit. We actually first prove existence for \emph{some} volume (that satisfies the stability conditions). The closedness aspect of this proof is technically challenging. The idea is to first prove estimates on two functions ($\psi$ and $q_{11}$ in the notation of Section \ref{sec:ODE}) and then extrapolate them to the other functions. But taking care of the boundary conditions also requires some work. \\
\indent  Then (in Section \ref{sec:volume}) we produce a family of solutions depending smoothly on $\lambda$, such that the volume is a continuous function of $\lambda$. We then prove that the volume approaches the limits of admissibility as $\lambda$ approaches certain $0$ or $\infty$, thus proving the existence of solutions with any prescribed admissible volume. For one of the limits of admissibility (Proposition \ref{prop:zerolimitlambda}), the proof is relatively simple. The other limit (Proposition \ref{prop:infinitylimitlambda}) is more complicated.\\
\indent The work in this paper is far from being the last word on the subject. In particular, one can consider $\mathcal{O}(N_1)\oplus \mathcal{O}(N_2)$ and $z^{l_1}, z^{l_2}$ satisfying a balancing condition mentioned in \cite{AlGaGa3}. It seems a bit difficult to reduce this system to a system of ODE. In another direction, it might interesting to compare the work of Yao \cite{Yao} with that of Section \ref{sec:volume}.\\

\emph{Acknowledgements}: The author thanks Chengjian Yao, Luis Alvarez-Consul, and Oscar Garcia-Prada for useful suggestions and encouragement. The author also thanks the anonymous referee for useful feedback. The author is partially supported by the DST FIST programme - 2021 [TPN
- 700661].
\section{Symmetries of the non-abelian vortex solution $H_0$}\label{sec:symmetriesofvortex}
\indent We first prove a quasiuniqueness result for the following non-abelian vortex equation on a holomorphic bundle $E$ over a compact Riemann surface $X$.
\begin{gather}
iF_H =\frac{\tau-\phi \otimes \phi^{*_H}}{2} \omega.
\label{eq:nonabelianvortex}
\end{gather}
\begin{lemma}\label{lem:uniqueness}
Let $\omega$ be a fixed smooth K\"ahler form. Suppose $H, He^g$ are two smooth solutions of \ref{eq:nonabelianvortex}, then $E$ is a holomorphic direct sum of $H$-orthogonal subbundles which are eigenbundles for $e^g$. The subbundle containing $\phi$ has eigenvalue $1$.
\end{lemma}
\begin{proof}
Let $H_t=H_0 e^{tg}$. Now we calculate as follows.
\begin{gather}
iF_{H_0e^g}-iF_{H_0} =\frac{\phi \otimes \phi^*_{H_0}-\phi \otimes \phi^*_{H_0e^g}}{2} \omega \nonumber \\
\Rightarrow \displaystyle \int_0^1 dt \left( \frac{d(iF_{H_t})}{dt}+\frac{1}{2}\frac{d(\phi \otimes \phi^*_{H_t})}{dt} \omega \right) =0\nonumber \\
\Rightarrow \int_0^1 dt \left (i\bar{\partial} \partial_t g + \frac{\phi\otimes \phi^*_{H_t}g}{2} \omega \right)=0\nonumber \\
\Rightarrow \int_0^1 dt \int_{X}\left( tr(i\partial_t g \wedge \bar{\partial} g)+\frac{tr(g\phi\otimes \phi^*_{H_t}g)}{2} \omega \right) =0.
\end{gather}
Thus $g$ is parallel and $g\phi=0$ identically (hence $e^g \phi=\phi$). Thus $e^g$ generates a holomorphic automorphism that is Hermitian and fixes $\phi$. Diagonalising this automorphism, we get a splitting of $E$ into an $H$-orthogonal direct sum of subbundles, one of which contains $\phi$.
\end{proof}
Now if $\varphi$ is a biholomorphism of $X$ that extends to an isomorphism of $E$, $\omega$ is $\varphi$-invariant, and $\phi$ is an equivariant section, then it is easy to see that if $H$ is a solution of Equation \ref{eq:nonabelianvortex}, so is $\varphi. H$. By the quasiuniqueness Lemma \ref{lem:uniqueness} we see that either $\varphi. H=H$ (when the bundle is $H$-indecomposable) or $H=\mathcal{H}\oplus\mathcal{K}$ and $E=E_1\oplus E_2$ where $E_1$ is the unique rank-$1$ subbundle containing $\phi$. In the latter case, $\varphi. H= \mathcal{H}\oplus \lambda \mathcal{K}$ for some $\lambda>0$. Applying this observation in conjunction with Remark \ref{rem:symmetries} we see that $H_0$ is $U(1)$-invariant.
\section{Formulation as a system of ODE}\label{sec:ODE}
\indent  Let $H_0$ be the solution of the vortex equation with $V\omega_{FS}$ as the K\"ahler metric and $H=H_0g$ where $g$ is an $H_0$-Hermitian positive-definite endomorphism of $E$. Let $H_0 = ( H_{FS}^N \oplus H_{FS}^N )g_0$ where $H_{FS}$ is the Fubini-Study metric on $\mathcal{O}(1)$. Let $\det(g)\det(g_0)=e^{-v}$. Note that
\begin{gather}
Ricci(\omega)=Ricc(V\omega_{FS})-\spbp \ln \left (\frac{\omega}{V\omega_{FS}} \right).
\label{eq:Ricciexp}
\end{gather}
Tracing both sides of the first equation in \ref{eq:nonabelianvortices}, we see that
\begin{gather}
-2 tr(iF_H)=\left(\Vert \phi \Vert_H^2 -2\tau \right)\omega.
\label{eq:tracingthefirstequation}
\end{gather}
We now substitute the expressions obtained from Equations \ref{eq:tracingthefirstequation} and \ref{eq:Ricciexp} in the second equation in \ref{eq:nonabelianvortices}. We arrive at the following expressions after using the fact that $H_0$ solves the vortex equation with $V\omega_{FS}$ as the background metric.
\begin{gather}
-\spbp \ln \left(\frac{\omega}{\omega_{FS}}\right)-2\alpha\spbp \Vert \phi \Vert_H^2 = 2\alpha \tau \mathrm{tr}(iF_H)-sRicci(V\omega_{FS})\nonumber \\
+(1-s)\alpha(-2\spbp +\tau V\omega_{FS})(\Vert \phi \Vert_{H_{s=0}}^2-2\tau) \nonumber \\
\Rightarrow 0=i\pbp \left(\ln\left(\frac{\omega}{V\omega_{FS}}\right)+2\alpha \Vert \phi \Vert_H^2-2\alpha (1-s)\Vert \phi \Vert_{H_{s=0}}^2  \right)\nonumber \\
-2(1-s)\alpha \tau \mathrm{tr}(iF_{H_0})  -2s \omega_{FS}+ 2\alpha \tau\mathrm{tr}(iF_H) \nonumber \\
\Rightarrow 0=i\pbp \left(\ln\left(\frac{\omega}{V\omega_{FS}}\right)+2\alpha \Vert \phi \Vert_H^2-2\alpha (1-s)\Vert \phi \Vert_{H_{s=0}}^2  +\frac{v}{N}+\frac{1}{N}\ln\det (g_0)\right)\nonumber \\
  -2s \omega_{FS}+s2\alpha \tau\mathrm{tr}(iF_{H_0}).
 \label{eq:somesimplifications}
\end{gather}
Let $u_0$ be a smooth function (with $\int u_0 \omega_{FS}=0$) satisfying $i\pbp u_0= 2\omega_{FS}-2\alpha \tau i \mathrm{tr}(F_{H_0})$. Substituting for $u_0$ in \ref{eq:somesimplifications} and solving for $\omega$, we arrive at the following expression.
\begin{gather}
\omega=2 \times 4^{1/N} \lambda_s \omega_{FS} e^{-2\alpha \Vert \phi \Vert_H^2-\frac{v}{N}},
\label{eq:foromega}
\end{gather}
where $\lambda_s =\frac{1}{2 \times 4^{1/N}}\det(g_0)^{-1/N} V\lambda^s e^{-(s-1)2\alpha \Vert \phi \Vert_{H_0}^2 +su_0}$, with $\lambda>0$ being a constant. As mentioned in the introduction, there exists a solution $H_0$ that is $U(1)$-symmetric. Since $u_0$ can be chosen (by averaging for instance) to be radially symmetric and hence $\lambda_s$ is radially symmetric on the sphere.\\
Substituting \ref{eq:foromega} in the first equation in \ref{eq:nonabelianvortices} we arrive at the following.
\begin{gather}
iF_H=4^{1/N} \lambda_s \omega_{FS}(\tau- \phi \otimes \phi^{*_H})e^{-2\alpha \Vert \phi \Vert_H^2-\frac{v}{N}}.
\label{eq:nonabelianvorticessimplifiedtooneequation}
\end{gather}
\indent In the usual trivialisation of $V$ on $\mathbb{C}$, $\det(H) (1+\vert z \vert^2)^{2N}=e^{-v}$. Therefore,
\begin{gather}
\frac{\partial}{\partial \bar{z}}\left(H^{-1} \frac{\partial H}{\partial z}\right)=4^{1/N}\lambda_s (\phi \otimes \phi^{*_H}-\tau)e^{-2\alpha \Vert \phi \Vert_H^2-\eta/N},
\label{eq:intheusualtriv}
\end{gather}
where $e^{-\eta}=\det(H)$. In order to extend smoothly across $\infty$, we want $H(1/w) \vert w \vert^{-2N}$ to extend as a smooth positive-definite matrix near $w=0$. The following lemma simplifies the extension condition.
\begin{lemma}
Let $H$ be a smooth Hermitian positive-definite matrix on $\mathbb{C}^*$ satisfying Equation \ref{eq:intheusualtriv}. Suppose $\det(H)$ is a radially symmetric function. If $H$ extends smoothly as a metric across $0$ and $\infty$, the following conditions are met.
\begin{gather}
\Vert \phi \Vert_{H}^2 (z) \leq \tau \ \forall \ z, \nonumber \\
\vert H(z) \vert \leq C \ \mathrm{near} \ z=0, \nonumber  \\
 \vert H(1/w) \vert \vert w^{-2N} \vert \leq C \ \mathrm{near} \ w=0, \nonumber \\
\left\Vert H^{-1} \frac{\partial H}{\partial z} \right\Vert_{L^{\infty}(B(0,1))}\leq C, \ \mathrm{and} \nonumber \\
\left \Vert H^{-1}\frac{\partial H}{\partial w} -\frac{N}{w} \right\Vert_{L^{\infty}(B(0,1))}\leq C.
\label{eq:iflimitsextension}
\end{gather}
Conversely, if the following conditions are met, then $H$ extends smoothly.
\begin{gather}
\Vert \phi \Vert_{H}^2 (z) \leq \tau \ \forall \ z, \nonumber \\
\lim_{z\rightarrow 0}  z   \frac{\partial \ln\det(H)}{\partial z} =0, \nonumber \\
\lim_{w\rightarrow 0}  w  \frac{\partial \ln\det(H)}{\partial w}  =2N, \nonumber \\
 \left \Vert  H^{-1} \frac{\partial H}{\partial z} \right\Vert_{L^2(B(0,1))}\leq C, \ \mathrm{and} \nonumber \\
 \left  \Vert H^{-1}\frac{\partial H}{\partial w} -\frac{N}{w} \right \Vert_{L^2(B(0,1))}\leq C.
\label{eq:limitsextension}
\end{gather}
\label{lem:extensionofH}
\end{lemma}
\begin{proof}
It is easy to see that if $H$ extends smoothly, then all the conditions in \ref{eq:iflimitsextension} except possibly the first, hold. The first condition holds thanks to the proof of Lemma \ref{lem:fundamentalq11estimate}.\\
 We now prove the other direction (the proof is a more complicated version of that of Lemma 3.2 in \cite{spruckyang}). We shall only prove extendability across $0$. The case of $\infty$ is similar.\\
\indent Tracing both sides of \ref{eq:intheusualtriv} we arrive at the following equation.
\begin{gather}
\frac{\partial^2 \eta}{\partial z\partial \bar{z}}=4^{1/N}\lambda_s (2\tau-\Vert \phi \Vert_H^2)e^{-2\alpha \Vert \phi \Vert_H^2-\eta/N}.
\label{eq:aftertrace}
\end{gather}
The given conditions \ref{eq:limitsextension} imply that that the right-hand-side of the distributional Equation \ref{eq:aftertrace} is an $L^{p}_{loc}$ function (for every $p$) on all of $\mathbb{C}$. Therefore, there is a $W^{2,p}$ for $p>>1$ (and hence $C^{1,\beta}$) solution $u$ of the following equation on the unit disc (with the smooth Dirichlet boundary condition of being equal to $\eta$ on the boundary).
\begin{gather}
\frac{\partial^2 u}{\partial z\partial \bar{z}}=4^{1/N}\lambda_s (2\tau-\Vert \phi \Vert_H^2)e^{-2\alpha \Vert \phi \Vert_H^2-\eta/N}.
\end{gather}
Therefore, $\eta-u$ is harmonic on the unit disc and satisfies $\lim_{z\rightarrow 0} \frac{\eta-u}{\ln \vert z \vert}=0$ (using the L' H\^opital rule for $\eta(\vert z \vert)$).  Lemma 3.1 in \cite{spruckyang} shows that $\eta-u$ extends in a harmonic manner across the origin. Thus, $\eta$ extends in a $C^{1,\beta}$ manner across the origin.\\
\indent Consider the distributional equation
\begin{gather}
\frac{\partial \theta}{\partial \bar{z}}=4^{1/N}\lambda_s (\phi \otimes \phi^{*_H}-\tau)e^{-2\alpha \Vert \phi \Vert_H^2-\eta/N}
\label{eq:distequationunitdisc}
\end{gather}
on the unit disc. Since $\bar{\partial}$ is elliptic in dimension-one and the right-hand-side is in $L^{\infty}$, there is a $W^{1,p}$ for large $p$ and hence $C^{0,\beta}$ solution to this equation on the unit disc.
Thus, $U=H^{-1} \frac{\partial H}{\partial z}-\theta$ is holomorphic in a punctured disc. Now $\displaystyle \int_{r<\vert z \vert<2r}  \Vert U \Vert \leq C r$ and hence $U$ extends holomorphically to the disc. This means that $U+\theta=H^{-1} \frac{\partial H}{\partial z}$ extends in a $W^{1,p}$ and $C^{0,\beta}$ manner to the centre. Thus, there exists a $C^{1,\beta}$ Hermitian solution $K$ to the elliptic equations $$\frac{\partial K}{\partial z}=K(U+\theta)$$ $$\frac{\partial K}{\partial \bar{z}}=(U+\theta)^{\dag}K$$ on a disc around the origin with $K=H$ on the boundary. Now $\Vert \nabla (K-H) \Vert\leq C\Vert K-H\Vert$ and $K-H=0$ on the boundary of the disc. Consider any line $\gamma(t)$ with constant speed from a point on the boundary of the disc to the origin. Along that line $\Vert\frac{d(K-H)}{dt} \Vert \leq C\Vert K-H \Vert$ and $(K-H)(0)=0$. Hence, $K-H=0$ throughout. This means that $H$ extends in a $C^{1,\beta}$ manner across the origin. By elliptic bootstrapping, $H$ is smooth.
\end{proof}
Define the Hermitian matrix $q$ as $q=P^{\dag}H P$ where
\begin{gather}
P=\left [\begin{array}{cc}
\phi_1 & -\phi_1 \\
\phi_2 & \phi_2
\end{array} \right ].
\label{eq:defofP}
\end{gather}
Clearly, $q$ is positive-definite away from $0$ and $\infty$. If $H$ satisfies \ref{eq:intheusualtriv}, then $q$ satisfies the following equation on $\mathbb{C}^{*}$ (noting that because $P$ changes the basis from the standard one ($e_1,e_2$) to the columns of $P$).
\begin{gather}
\frac{\partial}{\partial \bar{z}}\left(H^{-1} \frac{\partial H}{\partial z}\right)=P\frac{\partial}{\partial \bar{z}}\left(q^{-1} \frac{\partial q}{\partial z}\right)P^{-1} \nonumber \\
\Rightarrow \frac{\partial}{\partial \bar{z}}\left(q^{-1} \frac{\partial q}{\partial z}\right)= \lambda_s 4^{1/N} (P^{-1}\phi \otimes \phi^{\dag} H P-\tau)e^{-2\alpha \mathrm{tr}(\phi\otimes \phi^{\dag_H})-\eta/N} \nonumber \\
=\lambda_s 4^{1/N} (e_1\otimes e_1 q-\tau)e^{-2\alpha \mathrm{tr}(\phi\otimes \phi^{\dag_H})-\eta/N} \nonumber \\
=\lambda_s 4^{1/N} \left(\left[\begin{array}{cc}1 & 0 \\ 0 & 0 \end{array} \right] q-\tau Id\right ) e^{-2\alpha q_{11}} \det(q)^{-1/N} \vert\det(P)\vert^{-2/N}\nonumber \\
=\lambda_s \left(\left[\begin{array}{cc}1 & 0 \\ 0 & 0 \end{array} \right] q-\tau Id\right )e^{-2\alpha q_{11}-\psi/N-\ln \vert z \vert^2}, \nonumber \\
=\lambda_s \left(\left[\begin{array}{cc}q_{11}-\tau & q_{12}  \\ 0 & -\tau \end{array} \right]\right )e^{-2\alpha q_{11}-\psi/N-\ln \vert z \vert^2}
\label{eq:forqawayfrom0andinfinity}
\end{gather}
where $e^{-\psi}=\det(q)$. Assume that $q$ satisfies \ref{eq:forqawayfrom0andinfinity},
\begin{gather}
q_{11}\leq \tau, \nonumber \\
\lim_{z\rightarrow 0}  z\frac{\partial \ln \det(q)}{\partial z} = N , \nonumber \\
\lim_{w\rightarrow 0} w  \frac{\partial \ln \det (q)}{\partial w}  = -N, \nonumber \\
\left \Vert  P  q^{-1} \frac{\partial q}{\partial z} P^{-1} - \left [\begin{array}{cc} \frac{l}{z} &0\\ 0 & \frac{N-l}{z} \end{array} \right ]\right \Vert_{L^2(B(0,1))}\leq C, \ \mathrm{and} \nonumber \\
 \left \Vert P  q^{-1} \frac{\partial q}{\partial w} P^{-1} +\left [\begin{array}{cc} \frac{N-l}{w} &0\\ 0 & \frac{l}{w} \end{array} \right ] \right\Vert_{L^2(B(0,1))}\leq C.
\label{eq:boundaryconditions}
\end{gather}
\indent Since Equation \ref{eq:forqawayfrom0andinfinity} is symmetric, we can try to find a symmetric solution.
\begin{remark}\label{rem:symmetries}
More precisely, consider the action of $S^1$ on $\mathbb{P}^1$ given by rotations, i.e., $[X^0:X^1]\rightarrow [X^0:e^{\sqrt{-1}\theta}X^1]$ and extend it in two ways to $\mathcal{O}(-N)\subset \mathbb{P}^1\times \mathbb{C}^2$ as $([X^0:X^1],\mu (X^0,X^1))\rightarrow ([X^0:e^{\sqrt{-1}\theta}X^1], \mu e^{\sqrt{-1}k_i\theta} (X^0,e^{\sqrt{-1}\theta}X^1)$ for two integers $k_1, k_2$. If we choose $k_1, k_2$ such that $k_1-k_2= 2l-N$, then it is easy to see that a metric (which is a section of $E\otimes \bar{E}$) transforms in such a way that $q$ is radially symmetric.
\end{remark}
\indent Assume that $q$ is a function of $t=\ln\vert z \vert^2$. Then
\begin{gather}
(q^{-1}q')'=\lambda_s \left(\left[\begin{array}{cc}1 & 0 \\ 0 & 0 \end{array} \right] q-\tau Id\right )e^{-2\alpha q_{11}-\psi/N} \nonumber \\
=\lambda_s \left(\left[\begin{array}{cc}q_{11}-\tau & q_{12}  \\ 0 & -\tau \end{array} \right]\right )e^{-2\alpha q_{11}-\psi/N}
\label{eq:forqintermsofs}
\end{gather}
 The boundary conditions in terms of $t$ are as follows.
\begin{gather}
q_{11}\leq \tau, \nonumber \\
\lim_{t\rightarrow -\infty} \psi' =-N, \nonumber \\
\lim_{t\rightarrow \infty} \psi' =N, \nonumber \\
\left \Vert   P q^{-1} q' P^{-1} - \left [\begin{array}{cc} l &0\\ 0 & N-l \end{array} \right ]\right\Vert_{L^2(-\infty,0)} =0  , \nonumber \\
\left \Vert   P q^{-1}q' P^{-1} +\left [\begin{array}{cc} N-l &0\\ 0 & l \end{array} \right ] \right \Vert_{L^2(0,\infty)}  =0.
\label{eq:boundaryintermsofs}
\end{gather}
These boundary conditions do not appear to be radially symmetric at first glance, but the explicit expressions in the proof of Proposition \ref{prop:necessary} (namely, Equations \ref{eq:generalPWPinv} and \ref{eq:PdagqP}) show that indeed they are equivalent to radially symmetric boundary conditions. Since $\det(H)=\frac{\det(q)}{4 \vert z\vert^{2N}}$ is radial, by Lemma \ref{lem:extensionofH}, the corresponding $H$ is a solution of Equation \ref{eq:nonabelianvorticessimplifiedtooneequation}. \\
\indent Before proceeding further, we record the following \emph{a priori} estimate on $\Vert \phi \Vert_H^2 = q_{11}$.
\begin{lemma}
Along the continuity path \ref{eq:nonabelianvortices}, $q_{11}=\Vert \phi \Vert_H^2 \leq \tau$.
\label{lem:fundamentalq11estimate}
\end{lemma}
\begin{proof}
Consider the following Bochner-type identity.
\begin{gather}
\spbp \Vert \phi \Vert^2 = -i\langle \phi, F \phi \rangle + i \langle \nabla^{1,0} \phi, \nabla^{1,0} \phi \rangle,
\label{eq:thebochneridentity}
\end{gather}
At this maximum of $\Vert \phi \Vert^2$, by the maximum principle, $i\langle \phi, F \phi \rangle\leq 0$. Using this inequality in Equation \ref{eq:nonabelianvorticessimplifiedtooneequation}, we see that $q_{11}=\Vert \phi \Vert^2\leq \tau$.
\end{proof}
\section{Openness}\label{sec:continuityopenness}
\indent Consider the set $\mathcal{B}_1$ of $C^{4,\beta}$ $S^1$-invariant (as per Remark \ref{rem:symmetries}) metrics $H$ satisfying the normalisation condition
\begin{gather}
4N\pi=\displaystyle \int 4^{1/N}\lambda_s \omega_{FS}e^{-2\alpha \Vert \phi \Vert_H^2-\frac{v}{N}} (2\tau-\Vert \phi \Vert_H^2).
\label{eq:normalisationofH}
\end{gather}
Using Lemma \ref{lem:fundamentalq11estimate} and the implicit function theorem, we see that $\mathcal{B}_1$ is a Banach manifold.  Let $\mathcal{B}_2$ be the Banach manifold of $C^{2,\alpha}$ $S^1$-invariant $(1,1)$-form valued endomorphisms whose trace has a fixed integral. Consider the smooth map $$T(H)=\sqrt{-1}F_H-4^{1/N}\lambda_s \omega_{FS}(\tau- \phi \otimes \phi^{*_H})e^{-2\alpha \Vert \phi \Vert_H^2-\frac{v}{N}}.$$ We write $H_h = H \exp(h)$. Linearising it at $H=T^{-1}(0)$ (i.e., $h=0$) we see that the following equation holds.
\begin{gather}
\delta T(h)= \sqrt{-1}\bar{\partial}\partial_H h+4^{1/N}\lambda_s \omega_{FS}e^{-2\alpha \Vert \phi \Vert_H^2-\frac{v}{N}} \left(\phi \otimes \phi^{*_H} h +\left(\tau-\phi \otimes \phi^{*_H} \right)\left(\frac{\mathrm{tr}(h)}{N}+2\alpha \mathrm{tr}(\phi \otimes \phi^{*_H} h)\right) \right).
\label{eq:linofT}
\end{gather}
Since this operator is formally self-adjoint, to prove openness, it is enough by elliptic theory and the implicit function theorem on Banach manifolds to prove that the kernel is trivial. Indeed, if $\delta T(h)=0$, then multiplying by $h$ and integrating-by-parts we see that
\begin{gather}
0=\int \mathrm{tr}(\partial_H h \wedge \bar{\partial} h) \nonumber \\
+\int 4^{1/N}\lambda_s \omega_{FS}e^{-2\alpha \Vert \phi \Vert_H^2-\frac{v}{N}} \left(\mathrm{tr}h\phi \otimes \phi^{*_H} h +\left(\tau\mathrm{tr}(h)-\mathrm{tr}(h\phi \otimes \phi^{*_H}) \right)\left(\frac{\mathrm{tr}(h)}{N}+2\alpha \mathrm{tr}(\phi \otimes \phi^{*_H} h)\right) \right).
\end{gather}
We claim that the second integrand is non-negative, i.e.,
\begin{lemma}
\begin{gather}
\mathrm{tr}(h\phi \otimes \phi^{*_H} h) +\left(\tau\mathrm{tr}(h)-\mathrm{tr}(h\phi \otimes \phi^{*_H}) \right)\left(\frac{\mathrm{tr}(h)}{N}+2\alpha \mathrm{tr}(\phi \otimes \phi^{*_H} h)\right) \geq 0.
\end{gather}
\label{lem:nonnegativesecondintegrand}
\end{lemma}
\begin{proof}
Choose a trivialisation where $H$ is identity and diagonalise $\phi\otimes \phi^{*_H}$. It has eigenvalues $0$ and $\Vert \phi \Vert^2 \leq \tau$ (by Lemma \ref{lem:fundamentalq11estimate}). Denote the diagonal entries of $h$ in this trivialisation as $\lambda_1, \lambda_2$ respectively. The expression in the statement of the lemma is equal to the following.
\begin{gather}
\lambda_2^2 \Vert \phi \Vert^2+\frac{\tau}{N} (\lambda_1+\lambda_2)^2+\left(2\alpha \tau-\frac{1}{N} \right) (\lambda_1+\lambda_2)\lambda_2 \Vert \phi \Vert^2 -2\alpha \lambda_2^2 \Vert \phi \Vert^4 \nonumber \\
\geq \frac{\tau}{N} (\lambda_1+\lambda_2)^2 +\lambda_2^2 \Vert \phi \Vert^2 \left( 1-\frac{1}{N}\right) \nonumber\\
\geq 0.
\end{gather}
\end{proof}
Hence, $h$ is a parallel Hermitian endomorphism of $E$ and thus a constant matrix of functions. The symmetries of $h$ imply that it is a diagonal constant matrix. Moreover, Lemma \ref{lem:nonnegativesecondintegrand} shows that the integrand is zero throughout, including when at the origin. Hence $tr(h)=0 \ \Rightarrow \ h_{11}=-h_{22}$. We now see that the following equation holds (using Lemma \ref{lem:fundamentalq11estimate}).
\begin{gather}
0=h_{11}^2(\vert \phi_1 \vert^2+ \vert \phi_2 \vert^2)-2\alpha h_{11}^2(\vert \phi_1\vert^2-\vert \phi_2 \vert^2)^2 \nonumber \\
\geq h_{11}^2 (\vert \phi_1 \vert^2+ \vert \phi_2 \vert^2) \left(1-2\alpha \tau \right)
\geq 0,
\end{gather}
because $2\alpha \tau =\frac{1}{N}\leq 1$. Since the strict inequality holds at least one point (because $0<\Vert \phi \Vert^2 < \tau$ near the origin), $h_{11}=h_{22}=0$. Hence openness is proven.
\section{Closedness}\label{sec:continuityclosedness}
\indent The constant $C$ in this section is independent of $s\in[0,1]$ but can vary from line-to-line. We first estimate $\psi$. The equation satisfied by it is
\begin{gather}
\psi^{''}=\lambda_s (2\tau-q_{11})e^{-2\alpha q_{11}-\psi/N}.
\label{eq:forpsi}
\end{gather}
Since $q_{11}\leq \tau$, $\psi$ is convex. The boundary conditions clearly imply that $\psi'\rightarrow \pm N$ as $t\rightarrow \pm \infty$ respectively. Thus $e^{-\psi}$ decreases exponentially at $\pm \infty$ at a fixed rate of $N$. Let $t_0$ be the unique point (of global minimum) where $\psi'(t_0)=0$. Without loss of generality (by subtraction), assume that $t_0= 0$. Multiplying \ref{eq:forpsi} by $\psi'$ and integrating, we see that
\begin{gather}
N^2=2\int_{0}^{\infty} \lambda_s (2\tau-q_{11})e^{-2\alpha q_{11}-\psi/N} \psi' \geq \frac{1}{C} e^{-\psi(0)/N}, \ \mathrm{and} \nonumber \\
N^2 \leq C e^{-\psi(0)/N}.
\label{ineq:forpsizero}
\end{gather}
Thus $\frac{1}{C}\leq e^{-\psi(0)/N}\leq C$. At this juncture, we prove some properties of the solution.
\begin{proposition}
Suppose there is a smooth solution $q$ to \ref{eq:forqintermsofs} satisfying the boundary conditions \ref{eq:boundaryintermsofs}. Then $q$ satisfies the following properties.
\begin{enumerate}
\item $\Vert q \Vert \leq C_s e^{-lt}$ near $\infty$ and $\Vert q \Vert \leq C_s e^{lt}$ near $-\infty$.
\item Let $q^{-1}(0)q'(0)=\left[\begin{array}{cc} a & b \\ c & -a \end{array} \right ]$. Then $a,b,c$ are real. Moreover, $N-l>c\pm a >-l$, $c>0$, $q_{12}$ is purely real, and $q_{12}(0)<0$.
\item Suppose $s_n\rightarrow s$. Then passing to a subsequence, $q_n$ converges smoothly on compact sets to a positive-definite limit $q$. Moreover, $\psi'(\pm \infty)=\pm N$.
\end{enumerate}
\label{prop:necessary}
\end{proposition}
\begin{proof}
\begin{enumerate}
\item Note that
\begin{gather}
P^{-1}=\frac{1}{2z^N} \left[\begin{array}{cc} z^{N-l} & z^l \\ -z^{N-l} & z^{l} \end{array} \right ] \nonumber\\
(P^{\dag})^{-1}=\frac{1}{2\bar{z}^N} \left[\begin{array}{cc} \bar{z}^{N-l} &  -\bar{z}^{N-l} \\ \bar{z}^l & \bar{z}^{l} \end{array} \right ].
\label{eq:forPinverse}
\end{gather}
Thus,
\begin{gather}
(P^{\dag})^{-1}qP^{-1}=\frac{1}{4\vert z\vert^{2N}} \left[\begin{array}{cc} \bar{z}^{N-l} &  -\bar{z}^{N-l} \\ \bar{z}^l & \bar{z}^{l} \end{array} \right ] \left[\begin{array}{cc} q_{11} & q_{12} \\ \bar{q}_{12} & q_{22} \end{array} \right ] \left[\begin{array}{cc} z^{N-l} & z^l \\ -z^{N-l} & z^{l} \end{array} \right ] \nonumber \\
=\frac{1}{4\vert z\vert^{2N}} \left[\begin{array}{cc} \bar{z}^{N-l} &  -\bar{z}^{N-l} \\ \bar{z}^l & \bar{z}^{l} \end{array} \right ]\left[\begin{array}{cc} (q_{11} -q_{12})z^{N-l} & (q_{11} +q_{12})z^{l} \\ (\bar{q}_{12}-q_{22})z^{N-l} & (\bar{q}_{12}+q_{22})z^{l} \end{array} \right ] \nonumber \\
=\frac{1}{4\vert z\vert^{2N}} \left[\begin{array}{cc} (q_{11} -q_{12}-\bar{q}_{12}+q_{22})\vert z\vert^{2(N-l)} & (q_{11} +q_{12}-\bar{q}_{12}-q_{22})z^{l}\bar{z}^{N-l} \\ (q_{11} -q_{12}+\bar{q}_{12}-q_{22})\bar{z}^{l}z^{N-l}  & (q_{11} +q_{12}+\bar{q}_{12}+q_{22})\vert z\vert^{2l} \end{array} \right ].
\label{eq:PdagqP}
\end{gather}
From these expressions and \ref{eq:iflimitsextension}, it easily follows that $q$ decreases exponentially at the rate of at least $l$. Indeed, each entry of the matrix above decays exponentially at that rate (noting that $2l<N$). Thus, upon ``solving" for $q_{ij}$, we see that so does $q$.
\item Let $W=\left[\begin{array}{cc} u&v \\ w&x \end{array} \right]$. Then
\begin{gather}
PWP^{-1}=  \frac{1}{2z^N}\left[\begin{array}{cc} z^{l} & -z^l \\ z^{N-l} & z^{N-l} \end{array} \right ] \left[\begin{array}{cc} u & v \\ w & x \end{array} \right ] \left[\begin{array}{cc} z^{N-l} & z^l \\ -z^{N-l} & z^{l} \end{array} \right ]\nonumber \\
=\frac{1}{2z^N}\left[\begin{array}{cc} z^{l} & -z^l \\ z^{N-l} & z^{N-l} \end{array} \right ] \left[\begin{array}{cc} (u-v)z^{N-l} & (u+v)z^l \\ (w-x)z^{N-l} & (w+x)z^l \end{array} \right ] \nonumber \\
=\frac{1}{2}\left[\begin{array}{cc} u-v-w+x & (u+v-w-x)z^{2l-N} \\ (u-v+w-x)z^{N-2l} & u+v+w+x \end{array} \right ].
\label{eq:generalPWPinv}
\end{gather}
At this point we note that upon integration of \ref{eq:forqintermsofs}, the following equation is obtained.
\begin{gather}
q'(t)= q(t)q^{-1}(0)q'(0)+ q(t) \int_0^t \lambda_s \left(\left[\begin{array}{cc}q_{11}(u)-\tau & q_{12}(u)  \\ 0 & -\tau \end{array} \right]\right )e^{-2\alpha q_{11}(u)-\psi(u)/N}du.
\label{eq:uponintegrationforq}
\end{gather}
Using \ref{eq:generalPWPinv} and \ref{eq:uponintegrationforq} we arrive at the following.
\begin{gather}
Pq^{-1}q' P^{-1}-\frac{1}{2}\left[\begin{array}{cc} -b-c & (2a+b-c)z^{2l-N} \\ (2a-b+c)z^{N-2l} & b+c \end{array} \right ]\nonumber \\
=\frac{1}{2}\int_0^t \lambda_s \left[\begin{array}{cc} q_{11}(u)-2\tau-q_{12}(u) & (q_{11}(u)+q_{12}(u))z^{2l-N} \\ (q_{11}(u)-q_{12}(u))z^{N-2l} & q_{11}(u)-2\tau+q_{12}(u) \end{array} \right ]  e^{-2\alpha q_{11}(u)-\psi(u)/N}du.
\end{gather}
The boundary conditions coupled with exponential decay imply the following.
\begin{gather}
-N+l+\frac{b+c}{2}=\frac{1}{2}\int_0^{\infty}\lambda_s (q_{11}(u)-2\tau-q_{12}(u))e^{-2\alpha q_{11}(u)-\psi(u)/N}du \nonumber\\
-l-\frac{b+c}{2}=\frac{1}{2}\int_0^{\infty} \lambda_s (q_{11}(u)-2\tau+q_{12}(u))e^{-2\alpha q_{11}(u)-\psi(u)/N}du \nonumber \\
l+\frac{b+c}{2}=-\frac{1}{2}\int_{-\infty}^{0}\lambda_s (q_{11}(u)-2\tau-q_{12}(u))e^{-2\alpha q_{11}(u)-\psi(u)/N}du \nonumber\\
N-l-\frac{b+c}{2}=-\frac{1}{2}\int_{-\infty}^{0}\lambda_s (q_{11}(u)-2\tau+q_{12}(u))e^{-2\alpha q_{11}(u)-\psi(u)/N}du.
\label{eq:bounddiagonaltermsnecessary}
\end{gather}
Adding the first and the fourth equations, and subtracting the sum of the second and the third equations in \ref{eq:bounddiagonaltermsnecessary}, we see that the following identity holds.
\begin{gather}
\displaystyle \int_{-\infty}^{\infty} \lambda_s q_{12}(u)e^{-2\alpha q_{11}(u)-\psi(u)/N}du =0.
\label{eq:zerointegralofq12}
\end{gather}
For the off-diagonal terms, since $2l<N$,
\begin{gather}
-\frac{2a-b+c}{2}=\frac{1}{2}\int_0^{\infty} \lambda_s(q_{11}(u)-q_{12}(u))e^{-2\alpha q_{11}(u)-\psi(u)/N}du \nonumber \\
-\frac{2a+b-c}{2}=-\frac{1}{2}\int_{-\infty}^{0}\lambda_s (q_{11}(u)+q_{12}(u))e^{-2\alpha q_{11}(u)-\psi(u)/N}du
\label{eq:boundwhen2lleNnecessary}.
\end{gather}
In fact,
\begin{gather}
\left\vert \int_t^{\infty} \lambda_s(q_{11}(u)-q_{12}(u))e^{-2\alpha q_{11}(u)-\psi(u)/N}du \right\vert \leq C_s e^{(2l-N)t/2} \ as \ t\rightarrow \infty \nonumber \\
\left\vert \int_{-\infty}^{t}\lambda_s (q_{11}(u)+q_{12}(u))e^{-2\alpha q_{11}(u)-\psi(u)/N}du \right\vert \leq C_s e^{(N-2l)t/2} \ as \ t\rightarrow -\infty
\label{ineq:strongeroffdiagonalboundary}
\end{gather}
From the above equations and $q_{11}\leq \tau$, we see that $a$ and $c$ are real, and satisfy $N-l>c\pm a >-l$. Indeed,
\begin{gather}
a= -\frac{1}{4}\left(\int_0^{\infty} \lambda_s q_{11}(u)e^{-2\alpha q_{11}(u)-\psi(u)/N}du - \int_{-\infty}^{0} \lambda_s q_{11}(u)e^{-2\alpha q_{11}(u)-\psi(u)/N}du \right) \nonumber \\
-2N= \int_{-\infty}^{\infty} \lambda_s (q_{11}(u)-2\tau)e^{-2\alpha q_{11}(u)-\psi(u)/N}du\nonumber \\
\leq -\tau  \int_{-\infty}^{\infty} \lambda_s e^{-2\alpha q_{11}(u)-\psi(u)/N}du \nonumber \\
b-c=\frac{1}{2}\int_{-\infty}^{\infty} \lambda_s q_{11}(u)e^{-2\alpha q_{11}(u)-\psi(u)/N}du -\int_0^{\infty} \lambda_s q_{12}(u)e^{-2\alpha q_{11}(u)-\psi(u)/N}du \nonumber \\
b+c= N-2l -\int_0^{\infty} \lambda_s q_{12}(u)e^{-2\alpha q_{11}(u)-\psi(u)/N}du \nonumber \\
\Rightarrow c=\frac{N-2l}{2}-\frac{1}{4} \int_{-\infty}^{\infty} \lambda_s q_{11}(u)e^{-2\alpha q_{11}(u)-\psi(u)/N}du \nonumber \\
\Rightarrow c=N-l-\frac{\tau}{2}\int_{-\infty}^{\infty} \lambda_s e^{-2\alpha q_{11}(u)-\psi(u)/N}du
\label{eq:explicitaandc}
\end{gather}
At this juncture, we note that since the solvability of the non-abelian vortex equation implies that $\frac{\tau V}{2}<2N-2l$ \cite{AlGaGa3},
\begin{gather}
c=N-l-\frac{\tau}{4} V>0.
\label{ineq:cispositive}
\end{gather}
We have used the expression for $\omega$ from Equation \ref{eq:foromega}, and integrated it over the Riemann surface to bring $V$ into the picture in Equation \ref{ineq:cispositive}. Let $q_{12}(s)=x(s)+\sqrt{-1}y(s)$. Equation \ref{eq:uponintegrationforq} can be expanded as follows.
\begin{gather}
\frac{1}{\det(q)}\left[\begin{array}{cc}q_{22} & -q_{12} \\ -\bar{q}_{12} & q_{11}\end{array} \right] \left[\begin{array}{cc} q_{11}' & q_{12}' \\ \bar{q}_{12}' & q_{22}'\end{array} \right ] = \left[\begin{array}{cc} a & b \\ c & -a \end{array} \right ] + \displaystyle  \int_0^t \lambda_s\left[\begin{array}{cc}q_{11}(u)-\tau & q_{12}(u)  \\ 0 & -\tau \end{array} \right]e^{-2\alpha q_{11}(u)-\psi(u)/N}du \nonumber \\
\Rightarrow \frac{1}{\det(q)}\left[\begin{array}{cc}q_{22} q_{11}'-q_{12}\bar{q}_{12}' & q_{22}q_{12}'-q_{12}q_{22}' \\ -\bar{q}_{12}q_{11}'+q_{11}\bar{q}_{12}' & -\bar{q}_{12}q_{12}'+q_{11}q_{22}'\end{array} \right] = \left[\begin{array}{cc} a & b \\ c & -a \end{array} \right ] + \nonumber \\
 \displaystyle  \int_0^t \lambda_s\left[\begin{array}{cc}q_{11}(u)-\tau & q_{12}(u)  \\ 0 & -\tau \end{array} \right]e^{-2\alpha q_{11}(u)-\psi(u)/N}du.
\label{eq:theintegrodiffequationforq}
\end{gather}
Since $a,c$ are real, we see that the following hold.
\begin{gather}
yx'=xy' \nonumber \\
\Rightarrow yx(0)=xy(0) \nonumber \\
y q_{11}'=q_{11} y' \nonumber \\
\Rightarrow y q_{11}(0)=q_{11} y(0).
\label{eq:sinceaandcarereal}
\end{gather}
Moreover,
\begin{gather}
\frac{q_{11}x'-x q_{11}'}{\det(q)}=c \nonumber \\
\Rightarrow xq_{11}(0)=x(0)q_{11}+q_{11} q_{11}(0) c \displaystyle \int_0 ^t \frac{\det(q)(u)}{q_{11}^2(u)}du.
\label{eq:therealpartforc}
\end{gather}
There are two cases.
\begin{enumerate}
\item $y(0)=0$: In this case, $y=0$ identically and hence $b$ is real.
\item $y(0)\neq 0$: Equations \ref{eq:sinceaandcarereal} and \ref{eq:therealpartforc} imply that $c=0$. This contradicts \ref{ineq:cispositive}.
\end{enumerate}
Hence, $y(0)=0$ and therefore $y$ is identically zero. This fact means that $b$ and $q_{12}$ are real. From \ref{eq:therealpartforc} and \ref{eq:zerointegralofq12} we see that $x(0)<0$. Indeed, if $x(0)>0$, then $x>0$ and hence its integral against any positive measure cannot be zero (thus contradicting \ref{eq:zerointegralofq12}). If $x(0)=0$, then $x$ is identically zero and hence so is $q_{12}$. This means that $c=0$ which is a contradiction.
\item In terms of $q_{11}$, the Bochner-type identity \ref{eq:thebochneridentity} is equivalent to
\begin{gather}
q_{11}''=\lambda_s q_{11}(q_{11}-\tau)e^{-2\alpha q_{11}-\psi/N}+\frac{(q_{11}')^2}{q_{11}}+\frac{(q_{11}'q_{12}- q_{12}'q_{11})^2}{\det(q) q_{11}} \nonumber \\
\Rightarrow q_{11}''=\lambda_s q_{11}(q_{11}-\tau)e^{-2\alpha q_{11}-\psi/N}+\frac{(q_{11}')^2}{q_{11}}+\frac{c^2 e^{-\psi}}{q_{11}},
\label{eq:bochnerintermsofq}
\end{gather}
where the last equality follows from Equation \ref{eq:therealpartforc}. We simplify further.
\begin{gather}
(\ln(q_{11}))''=\lambda_s (q_{11}-\tau)e^{-2\alpha q_{11}-\psi/N}+c^2\frac{e^{-\psi}}{q_{11}^2}.
\label{eq:simplifyforq11derivativesfrombochner}
\end{gather}
From \ref{eq:simplifyforq11derivativesfrombochner} and \ref{eq:forpsi} we see that
\begin{gather}
\left(\ln(q_{11})+\frac{\psi}{2}\right)''=\lambda_s \frac{q_{11}}{2}e^{-2\alpha q_{11}-\psi/N}+c^2\frac{e^{-\psi}}{q_{11}^2} >0.
\label{eq:convexityofq11pluspsiby2}
\end{gather}
Thus the $\displaystyle \lim_{t\rightarrow \infty} \left(\ln(q_{11})+\frac{\psi}{2}\right)'=A_+$ limit exists (as possibly an extended real number) and likewise for $A_{-}$ at $-\infty$. Since the limit of $\psi'$ exists, so does that of $(\ln(q_{11}))'$. Since $q_{11}\leq \tau$, this limit is $\leq 0$ at $\infty$ and $\geq 0$ at $-\infty$. Now we consider a sequence of $q_n$ corresponding to $s_n$. Since $\psi_n ''$ is bounded uniformly and so is $\psi_n(0)$, passing to a subsequence $\psi_n\rightarrow \psi$ in $C^{1,\beta}_{loc}$. Now
\begin{gather}
\psi'(t)\geq \psi_n'(1)\geq \frac{1}{C} \ \forall \ t\geq 1 \nonumber \\
\psi'(t)\leq  \psi_n'(-1)\leq \frac{-1}{C} \ \forall \ t \leq -1,
 \label{ineq:fundestimateonpsindashdct}
\end{gather}
 because $\psi_n''$ is bounded below on any compact set.
Now
\begin{gather}
-N\leq -\frac{\psi'_n}{2}+(A_-)_n\leq (\ln(q_{11})_n)'\leq -\frac{\psi'_n}{2}+(A_+)_n \leq N,
\label{ineq:estimateonlnq11dash}
\end{gather}
and hence
\begin{gather}
(q_{11})_n(0)e^{-Nt}\leq (q_{11})_n(t)\leq (q_{11})_n(0)e^{Nt}.
\label{ineq:q11intermsof110}
\end{gather}
 Thus $q_{11}$ is bounded in $C^{1}_{loc}$ and hence a subsequence converges to $q_{11}$ in $C^{0,\beta}_{loc}$. If $(q_{11})(0)$ is $0$, then $q_{11}$ is $0$ identically. From the expression for $c$ in \ref{eq:explicitaandc} we see that using the dominated convergence theorem that $c_n\rightarrow c=\frac{N-2l}{2}\neq 0$. Integrating both sides of \ref{eq:convexityofq11pluspsiby2} we see that $\displaystyle \int_{-\infty}^{\infty} \frac{e^{-\psi_n}}{(q_{11})_n^2} \leq C$. Thus $(q_{11})_n(0)$ cannot approach $0$. Therefore, $q_{11}>0$ everywhere. Equation \ref{eq:bochnerintermsofq} shows that $(q_{11})_n$ is $C^{2}_{loc}$. Thus $q_{11}$ is $C^{1,\beta}$ and $\psi$ is thus in $C^{3,\beta}$ and we can bootstrap to conclude that $q_{11}$ and $\psi$ are smooth. The dominated convergence theorem applied to
\begin{gather}
\displaystyle \psi_n'(\pm\infty) =\pm\int_0^{\infty} \psi_n'' = \pm\int_0 ^{\infty} \lambda_{s_n} (2\tau-(q_{11})_n)e^{-2\alpha (q_{11})_n-\psi_n/N}
\end{gather}
shows that $\psi'(\pm \infty)=\pm N$. We also see (by integrating \ref{eq:convexityofq11pluspsiby2}) that in the limit
\begin{gather}
\displaystyle \int_{-\infty}^{\infty} c^2\frac{e^{-\psi}}{q_{11}^2}<\infty.
\label{ineq:forc2integrallimit}
\end{gather}
Moreover, since $e^{-\psi}$ decays exponentially,
\begin{gather}
\displaystyle \int_{-\infty}^{\infty} e^{-2\alpha q_{11}-\psi/N}<\infty
\label{ineq:finitevolumelimit}
\end{gather}
Using \ref{ineq:finitevolumelimit}, \ref{eq:therealpartforc}, and \ref{eq:zerointegralofq12}, we see that $x_n(0)=(q_{12})_n(0)$ is bounded and hence (upto a subsequence) $x_n$ converges smoothly on compact sets to a limit $x$. Since $e^{-\psi_n}(0)+x_n(0)^2=(q_{11})_n(0)(q_{22})_n(0)$, we see using Equation \ref{eq:theintegrodiffequationforq}, that $(q_{22})_n'(0)$ is bounded. Moreover, using \ref{eq:theintegrodiffequationforq}, we see that
\begin{gather}
(q_{22})_n' -\frac{(q_{11})_n'}{(q_{11})_n}(q_{22})_n=-\frac{\det(q_n)}{(q_{11})_n}\left(2a_n+ \displaystyle \int_0^ t \lambda_{s_n} (q_{11})_n e^{-2\alpha (q_{11})_n-\psi_n/N}\right).
\label{eq:theeqforq22n}
\end{gather}
We can explicitly solve Equation \ref{eq:theeqforq22n} to conclude bounds on $(q_{22})_n$ and hence (up to a subsequence) smooth convergence on compact sets to a limit $q_{22}$. Moreover, $q_{22}>0$ (because $\det(q)+x^2=q_{11}q_{22}$) and since $\det(q)>0$, $q$ is positive-definite.
\end{enumerate}
\end{proof}
\indent We now investigate the boundary behaviour of the limiting $q$ obtained from the last part of Proposition \ref{prop:necessary}. \\
We first observe (using the the dominated convergence theorem applied to the explicit expression for $a_n$ (Equation \ref{eq:explicitaandc})) that
\begin{gather}
2a+ \displaystyle \int_0^ {\infty} \lambda_{s} q_{11} e^{-2\alpha q_{11}-\psi/N}>0 \nonumber \\
2a+\displaystyle \int_0^ {-\infty} \lambda_{s} q_{11} e^{-2\alpha q_{11}-\psi/N}<0.
\label{ineq:onaplusforq22}
\end{gather}
Using \ref{eq:theintegrodiffequationforq} and \ref{ineq:onaplusforq22} we see that
\begin{gather}
(q_{22})' -\frac{(q_{11})'}{(q_{11})}(q_{22})=-\frac{\det(q)}{(q_{11})}\left(2a+ \displaystyle \int_0^ t \lambda_{s} q_{11} e^{-2\alpha q_{11}-\psi/N}\right) \nonumber \\
\leq 0 \ \mathrm{if} \ t\geq K, \nonumber \\
\geq 0 \ \mathrm{if} \ t\leq -K.
\label{eq:theeqforq22}
\end{gather}
Thus for large enough $K$,
\begin{gather}
q_{11} \frac{q_{22}(-K)}{q_{11}(-K)} \leq q_{22}(t) \ \forall \ t\leq -K, \\
q_{22}(t)\leq q_{11} \frac{q_{22}(K)}{q_{11}(K)} \ \forall \ t\geq K.
\label{ineq:comparingq11andq22}
\end{gather}
In fact, since $2a+ \displaystyle \int_0^ K \lambda_{s} q_{11} e^{-2\alpha q_{11}-\psi/N}>0$ we see that for all large enough $n$ (independent of $K$), $2a_n+ \displaystyle \int_0^ K \lambda_{s_n} (q_{11})_n e^{-2\alpha (q_{11})_n-\psi_n/N}>0$. Since this function is increasing,
\begin{gather}
(q_{11})_n \frac{(q_{22})_n(-K)}{(q_{11})_n(-K)} \leq (q_{22})_n(t) \ \forall \ t\leq -K, \\
(q_{22})_n(t) \leq (q_{11})_n \frac{(q_{22})_n(K)}{(q_{11})_n(K)} \ \forall \ t\geq K.
\label{ineq:comparingq11nandq22n}
\end{gather}
This observation implies that
\begin{gather}
(q_{11})_n^2 \frac{(q_{22})_n(-K)}{(q_{11})_n(-K)} \leq x_n^2 (t) \ \forall \ t\leq -K, \\
x_n^2 (t) \leq (q_{11})_n^2 \frac{(q_{22})_n(K)}{(q_{11})_n(K)} \ \forall \ t\geq K.
\label{ineq:comparingxnq11n}
\end{gather}
These observations imply that the dominated convergence theorem can be used with impunity to justify that the integral expression of  $b_n$ converges to that of $b$ (for instance). We now prove the following lemma.
\begin{lemma}\label{lem:finitenessofintegralofdetqbyq112}
\begin{gather}
\displaystyle \int_{-\infty}^{\infty}\frac{\det(q)}{q_{11}^2}<\infty.
\label{ineq:aftercneq0}
\end{gather}
Hence $A_+\geq 0$, $A_-\leq 0$.
\end{lemma}
\begin{proof}
We differentiate $e^{-\psi}=q_{11} q_{22}-x^2$ to get the following (using \ref{eq:therealpartforc} and \ref{eq:theeqforq22n}).
\begin{gather}
-e^{-\psi} \psi' = q_{11}'q_{22}+q_{22}'q_{11}-2xx' \nonumber \\
= 2q_{11}'q_{22} -\det(q) \left(2a+\int_0 ^t \lambda_s q_{11}e^{-2\alpha q_{11}-\psi/N}  \right)-2\frac{x^2}{q_{11}}q_{11}' -2xc\frac{\det(q)}{q_{11}} \nonumber \\
\Rightarrow \psi' = -2(\ln(q_{11}))' + \left(2a+\int_0 ^t \lambda_s q_{11}e^{-2\alpha q_{11}-\psi/N}  \right)+2\frac{xc}{q_{11}}.
\label{eq:forpsidashintermsofq11}
\end{gather}
Now there are two possibilities.
\begin{enumerate}
  \item $c= 0$: In this case, taking the limits as $t\rightarrow \pm \infty$ in \ref{eq:forpsidashintermsofq11} and using \ref{ineq:aftercneq0} we see that $A_+>0, A_-<0$.
  \item $c\neq 0$: Integrating Equation \ref{eq:convexityofq11pluspsiby2} we see that $\displaystyle \int_{-\infty}^{\infty}\frac{\det(q)}{q_{11}^2}<\infty$.
\end{enumerate}
\end{proof}
Using Lemma \ref{lem:finitenessofintegralofdetqbyq112} we see that the following estimates hold.
\begin{gather}
-\frac{N}{2}\leq  (\ln(q_{11}))'(\infty)\leq 0 \nonumber \\
0\leq (\ln(q_{11}))'(-\infty)\leq \frac{N}{2}.
\label{ineq:onewayboundforderiusingAplusminus}
\end{gather}
\begin{remark}\label{rem:preciseboundsonlogarithmicderivative}
In fact, integrating Equation \ref{eq:simplifyforq11derivativesfrombochner} we see that
\begin{gather}
\ln((q_{11})_n)'(\infty)-\ln((q_{11})_n)'(0)=\displaystyle \int_0 ^{\infty} \lambda_{s_n} ((q_{11})_n-\tau)e^{-2\alpha (q_{11})_n-\psi_n/N}+c_n^2\int_0 ^{\infty}\frac{e^{-\psi_n}}{(q_{11})_n^2} \nonumber \\
\ln((q_{11})_n)'(-\infty)-\ln((q_{11})_n)'(0)=\int_0 ^{-\infty} \lambda_{s_n} ((q_{11})_n-\tau)e^{-2\alpha (q_{11})_n-\psi_n/N}+c_n^2 \int_0 ^{-\infty}\frac{e^{-\psi_n}}{(q_{11})_n^2}.
\label{eq:integratingseconderioflnq11}
\end{gather}
Using the decay of $(q_{11})_n$ as per Proposition \ref{prop:necessary}, the dominated convergence theorem (for the first integral) and the Fatou lemma (for the second integral), we conclude that
\begin{gather}
-\frac{N}{2}\leq  (\ln(q_{11}))'(\infty)\leq -l \nonumber \\
l\leq (\ln(q_{11}))'(-\infty)\leq  \frac{N}{2}.
\label{ineq:boundsondecayoflnq11dash}
\end{gather}
Now we can easily conclude (using \ref{eq:therealpartforc} and \ref{ineq:comparingq11andq22}) that $q$ decays at least at a rate of $\frac{l}{2}$ at $\pm \infty$.
\end{remark}
\indent The dominated convergence theorem along with the expressions \ref{eq:bounddiagonaltermsnecessary} show that the diagonal parts of the boundary conditions \ref{eq:boundaryintermsofs} hold. For the off-diagonal parts, if we show that the following estimates hold, then we the off-diagonal $L^2$ boundary conditions will be met.
\begin{gather}
\displaystyle \left \Vert e^{(N-2l+1/100)t/2} \left(\frac{2a-b+c}{2}+\frac{1}{2}\int_0^{t} \lambda_s (q_{11}(u)-x(u))e^{-2\alpha q_{11}(u)-\psi(u)/N}du \right) \right \Vert_{L^{\infty}(0,\infty)} \leq C \nonumber \\
\left\Vert e^{-(N-2l+1/100)t/2} \left(\frac{2a+b-c}{2}- \frac{1}{2}\int_{t}^0\lambda_s (q_{11}(u)+x(u))e^{-2\alpha q_{11}(u)-\psi(u)/N}du \right) \right \Vert_{L^{\infty}(-\infty,0)} \leq C.
\label{ineq:exponentialneedtoshowboundary}
\end{gather}

Estimates \ref{eq:boundwhen2lleNnecessary} and the dominated convergence theorem reduce \ref{ineq:exponentialneedtoshowboundary} to the following.
\begin{gather}
\displaystyle \left \Vert e^{(N-2l+1/100)t/2} \int_t^{\infty} \lambda_s (q_{11}(u)-x(u))e^{-2\alpha q_{11}(u)-\psi(u)/N}du  \right \Vert_{L^{\infty}(0,\infty)} \leq C \nonumber \\
\left\Vert e^{-(N-2l+1/100)t/2} \int_{-\infty}^t\lambda_s (q_{11}(u)+x(u))e^{-2\alpha q_{11}(u)-\psi(u)/N}du  \right \Vert_{L^{\infty}(-\infty,0)} \leq C.
\label{ineq:ttoinftyexponentialneedtoshowboundary}
\end{gather}
The explicit expression for $x$ (Equation \ref{eq:therealpartforc}) shows that $\frac{q_{11}\pm x}{q_{11}}$ are strictly monotonic. Thus they have a sign at $\pm \infty$. This observation combined with $0<q_{11}\leq \tau$ and $\psi'(\pm\infty)=\pm N$ converts \ref{ineq:ttoinftyexponentialneedtoshowboundary} to the following.
\begin{gather}
\displaystyle \left \Vert e^{(N-2l+1/100)t/2} \int_t^{\infty} q_{11}(u)\left(1-\frac{x(0)}{q_{11}(0)}-c\displaystyle \int_0^u \frac{\det(q)(\nu)}{q_{11}^2(\nu)}d\nu \right)e^{-(1-1/1000N)u}du  \right \Vert_{L^{\infty}(0,\infty)} \leq C \nonumber \\
\left\Vert e^{-(N-2l+1/100)t/2} \int_{-\infty}^t q_{11}(u)\left(1+\frac{x(0)}{q_{11}(0)}+c\displaystyle \int_0^u \frac{\det(q)(\nu)}{q_{11}^2(\nu)}d\nu \right )e^{(1-1/1000N)u}du  \right \Vert_{L^{\infty}(-\infty,0)} \leq C.
\label{ineq:explicitxttoinftyexponentialneedtoshowboundary}
\end{gather}
We will only consider the first inequality in \ref{ineq:explicitxttoinftyexponentialneedtoshowboundary}. The proof of the second estimate is similar. The rough idea is that for all $n$, there is an $L^{\infty}$ bound (one that is better than the one required above). The integrands are roughly exponential and hence we have control over the exponents. There are two cases:
\begin{enumerate}
  \item $(\ln(q_{11}))'(\infty)=-\frac{N}{2}$: In this case, $q_{11}\leq C e^{-\left(\frac{N}{2}-\frac{1}{100000000} \right)t}$. Thus
\begin{gather}
\displaystyle \left \vert e^{(N-2l+1/100)t/2} \int_t^{\infty} q_{11}(u)\left(1-\frac{x(0)}{q_{11}(0)}-c\displaystyle \int_0^u \frac{\det(q)(\nu)}{q_{11}^2(\nu)}d\nu \right)e^{-(1-1/10000000N^5)u}du  \right \vert \nonumber \\
\leq C e^{(N-2l+1/100)t/2} \int_{t}^{\infty} e^{-\left(\frac{N}{2}-\frac{1}{1000}+1 \right)u} du \nonumber \\
\leq C e^{(-2l-1)t/2}.
\end{gather}
  \item $(\ln(q_{11}))'(\infty)>-\frac{N}{2}$: In this case, $A_+=\frac{N}{2}+\ln(q_{11})'(\infty)>0$. By the dominated convergence theorem and the Fatou lemma,
      \begin{gather}
      (A_{+})_n\geq \frac{N}{2}+\ln(q_{11})'(\infty)-\epsilon>A_+\left(1-\frac{1}{10000N^3} \right)>0
      \label{ineq:onAplusn}
      \end{gather}
       for small $\epsilon=\frac{A_{+}}{10000000N^4}$ and large enough $n$. Using the dominated convergence theorem again, we see that
\begin{gather}
\lim_{n\rightarrow \infty} (\ln((q_{11})_n))'(\infty)=(\ln(q_{11}))'(\infty) \nonumber \\
\lim_{n\rightarrow \infty} (A_+)_n=A_+.
\label{eq:limitsafterDCT}
\end{gather}
Now we use the fact that for all $n$, 
\begin{gather}
\displaystyle \left \Vert e^{(N-2l+1-\frac{1}{1000})t/2} \int_t^{\infty} (q_{11})_n(u)\left(1-\frac{x_n(0)}{(q_{11})_n(0)}-c_n\displaystyle \int_0^u \frac{\det(q_n)(\nu)}{(q_{11})_n^2(\nu)}d\nu \right)e^{-(1+\frac{1}{1000})u}du  \right \Vert_{L^{\infty}(0,\infty)}  \leq C_n.
\label{ineq:fornbounds}
\end{gather}
We now have two subcases depending on the value of $U_n=1-\frac{x_n(0)}{(q_{11})_n(0)}-c_n\displaystyle \int_0^{\infty} \frac{\det(q_n)(\nu)}{(q_{11})_n^2(\nu)}d\nu$.
\begin{enumerate}
\item For all but finitely many $n$, $U_n =0$: Thus $U=\displaystyle \lim_{n\rightarrow \infty}U_n=0$.  We now see that the following inequalities hold.
\begin{gather}
\displaystyle \left \vert e^{(N-2l+1/100)t/2} \int_t^{\infty} q_{11}(u)\left(1-\frac{x(0)}{q_{11}(0)}-c\displaystyle \int_0^u \frac{\det(q)(\nu)}{q_{11}^2(\nu)}d\nu \right)e^{-(1-1/1000N)u}du  \right \vert \nonumber \\
= \left \vert e^{(N-2l+1/100)t/2} \int_t^{\infty} q_{11}(u)c\displaystyle \int_u^{\infty} \frac{\det(q)(\nu)}{q_{11}^2(\nu)}d\nu e^{-(1-1/1000N)u}du  \right \vert \nonumber \\
\leq C e^{(N-2l+1/100)t/2} e^{((\ln(q_{11}))'_{\infty}-2A_+-1+\frac{1}{300})t} \leq C.
\end{gather}
\item For infinitely many $n$, $U_n\neq 0$: As a consequence of this assumption, for all sufficiently large $n$, we have the following bound.
\begin{gather}
\left \vert e^{(N-2l+1-\frac{1}{1000})t/2} \int_t^{\infty} (q_{11})_n(u)e^{-(1+\frac{1}{1000})u}du   \right \vert \leq C_n.
\end{gather}
Now $(q_{11})_n\geq \frac{1}{C} e^{\left((\ln((q_{11})_n))'(\infty)-\frac{1}{100000N^6}\right)t}$. Hence,
\begin{gather}
e^{(N-2l+1-\frac{1}{1000})t/2} e^{\left((\ln((q_{11})_n))'(\infty)-1-\frac{1}{400}\right)t}\leq C_n
\end{gather}
Thus $N/2-l+(\ln((q_{11})_n))'(\infty) \leq \frac{1}{200}$. Thus taking limits, we see that
\begin{gather}
N/2-l+(\ln((q_{11})))'(\infty) \leq \frac{1}{200} \nonumber \\
\Rightarrow \displaystyle \left \vert e^{(N-2l+1/100)t/2} \int_t^{\infty} q_{11}(u)\left(1-\frac{x(0)}{q_{11}(0)}-c\displaystyle \int_0^u \frac{\det(q)(\nu)}{q_{11}^2(\nu)}d\nu \right)e^{-(1-1/1000N)u}du  \right \vert \nonumber \\
\leq C e^{(N-2l+1/100)t/2} \int_t^{\infty} q_{11}(u) e^{-(1-1/1000N)u}du \nonumber \\
\leq C e^{t\left(N/2-l+1/100+\ln(q_{11})'(\infty)-1 \right)} \leq C.
\end{gather}
\end{enumerate}
\end{enumerate}
\indent Therefore, we have managed to show that if $s_n\rightarrow s$, passing to a subsequence $q_n\rightarrow q$ smoothly on compact sets and $q$ satisfies the right boundary conditions. Thus the set of $s$ for which there is a solution is closed. Hence Equation \ref{eq:nonabelianvortices} has a symmetric smooth solution at $s=1$.
\section{Solutions for all admissible volumes}\label{sec:volume}
\indent So far, for any given $\lambda>0$, we have produced a solution. Now we will produce a solution for any given admissible volume. We first prove the following uniqueness result.
\begin{proposition}\label{prop:uniquenessforgivenlambda}
Given $\lambda>0$, there is a unique smooth solution to Equation \ref{eq:nonabelianvorticessimplifiedtooneequation}.
\end{proposition}
\begin{proof}
Suppose there are two solutions, $H_1=H_0 g_1$ and $H_2= H_0 g_2 = H_1 g_1^{-1} g_2= H_1 g=H_1 \exp (h)$. Let $H_t = H_1 \exp((t-1)h)$ (so that $H_{t=1}=H_1$, $H_{t=2}=H_2$, and $H_t=H_0 g_t$). Then
\begin{gather}
i\bar{\partial}(g^{-1} \partial_{H_1} g) =4^{1/N} \lambda_{s=1} \omega_{FS}\left((\tau- \phi \otimes \phi^{*_{H_2}})e^{-2\alpha \Vert \phi \Vert_{H_2}^2-\frac{v_2}{N}}-(\tau- \phi \otimes \phi^{*_{H_1}})e^{-2\alpha \Vert \phi \Vert_{H_1}^2-\frac{v_1}{N}} \right) \nonumber \\
\Rightarrow \displaystyle \int_1^2 dt \frac{d}{dt}i\bar{\partial}(e^{-(t-1)h} \partial_{H_1} e^{(t-1)h})
= 4^{1/N} \lambda_{s=1} \omega_{FS} \displaystyle \int_1 ^2 dt \frac{d}{dt}\left((\tau- \phi \otimes \phi^{*_{H_t}})e^{-2\alpha \Vert \phi \Vert_{H_t}^2-\frac{v_t}{N}} \right) \nonumber \\
\Rightarrow
0= \displaystyle \int_1 ^2 dt i \bar{\partial}\partial_{H_t} h  \nonumber \\
+ 4^{1/N} \lambda_{s=1} \omega_{FS} \displaystyle \int_1 ^2 dte^{-2\alpha \Vert \phi \Vert_{H_t}^2-\frac{v_t}{N}} \Bigg(\phi\otimes \phi^{*_{H_t}}h+\left(\tau- \phi \otimes \phi^{*_{H_t}}\right)\left(2\alpha \mathrm{tr}(h\phi \otimes \phi^{*_{H_t}})+\frac{\mathrm{tr}(h)}{N} \right) \Bigg).
\label{eq:H2intermsofH1}
\end{gather}
Note that by convexity, $\Vert \phi \Vert_{H_t}^2\leq \tau$. Just as in Section \ref{sec:continuityopenness}, we multiply by $h$, take trace, and integrate-by-parts. The proof of Lemma \ref{lem:nonnegativesecondintegrand} shows that the right-hand-side of Equation \ref{eq:H2intermsofH1} is non-negative and is hence identically zero. Thus $h$ is a parallel symmetric endomorphism, and hence a constant matrix. The reasoning after the proof of Lemma \ref{lem:nonnegativesecondintegrand} shows that $h=0$.
\end{proof}
Proposition \ref{prop:uniquenessforgivenlambda} along with the proof of openness using the implicit function theorem shows that the solution $q$ varies smoothly in $\lambda$. Now we study the volume $V_{\lambda}$ of the K\"ahler metric $\omega_{\lambda}$ as $\lambda$ approaches $0$. To this end, define $\tilde{\psi}_{\lambda}=-N\ln \lambda+\psi_{\lambda}$.
\begin{proposition}\label{prop:zerolimitlambda}
$\displaystyle \lim_{\lambda\rightarrow 0^+} V_{\lambda} = \frac{4N-4l}{\tau}$.
\end{proposition}
\begin{proof}
The arguments in \ref{ineq:forpsizero} show that $\frac{1}{C}\leq \lambda e^{-\psi_{\lambda}(0)/N} \leq C$.   Since $\tilde{\psi}_{\lambda}'' =\psi''_{\lambda}$, $\tilde{\psi}_{\lambda}$ is strictly convex. Therefore, we see that $e^{-\tilde{\psi}_{\lambda}}$ decays to $0$ exponentially uniformly as $\lambda\rightarrow 0$, and that $\psi_{\lambda}\rightarrow -\infty$ pointwise. Using this observation, the facts that $\ln((q_{11})_{\lambda})'(\infty)\leq -l, \ln((q_{11})_{\lambda})'(-\infty)\geq l$, and integrating \ref{eq:simplifyforq11derivativesfrombochner} we see that
\begin{gather}
c_{\lambda}^2 \displaystyle \int_{-\infty}^{\infty} \frac{e^{-\psi_{\lambda}}}{(q_{11})_{\lambda}^2}<\infty.
\label{ineq:oncsquareetc}
\end{gather}
Since $(q_{11})_{\lambda}\leq \tau$, the Fatou lemma shows that $\displaystyle \int_{-\infty}^{\infty} \frac{e^{-\psi_{\lambda}}}{(q_{11})_{\lambda}^2}\rightarrow \infty$ and hence $c_{\lambda}\rightarrow 0$. Using the explicit formula for $c_{\lambda}$ (Equation \ref{eq:explicitaandc}) we see that
\begin{gather}
\displaystyle \lim_{\lambda\rightarrow 0^{+}} \frac{\tau}{2}\int_{-\infty}^{\infty} \lambda_{s=1} e^{-2\alpha (q_{11})_{\lambda}(u)-\psi_{\lambda}(u)/N}du  = N-l.
\end{gather}
Thus
\begin{gather}
\displaystyle \lim_{\lambda\rightarrow 0^{+}} V_{\lambda} = \lim_{\lambda\rightarrow 0^{+}} \int_{-\infty}^{\infty} 2\lambda_{s=1} e^{-2\alpha (q_{11})_{\lambda}(u)-\psi_{\lambda}(u)/N}du = \frac{4(N-l)}{\tau}.
\end{gather}
\end{proof}
Next we study the limit of the volume as $\lambda\rightarrow \infty$.
\begin{proposition}\label{prop:infinitylimitlambda}
$\displaystyle \lim_{\lambda\rightarrow \infty} V_{\lambda} = \frac{2N}{\tau}$.
\end{proposition}
\begin{proof}
As in the proof of Proposition \ref{prop:zerolimitlambda} we see that $-C\leq e^{-\tilde{\psi}_{\lambda}}\leq C$ and $e^{-\psi_{\lambda}}\rightarrow 0$ as $\lambda\rightarrow \infty$. Moreover, consider any sequence of $\lambda_n\rightarrow \infty$. There exists a subsequence (which we continue to denote by $\lambda_n$) such that $\tilde{\psi}_{\lambda_n}$ converges smoothly on compact sets to $\tilde{\psi}_{\infty}$. By \ref{ineq:estimateonlnq11dash} we see that $(q_{11})_n'$ is bounded. Consider any $W_{loc}^{1,p}$ (and hence $C^{0,\beta}_{loc}$) convergent subsequence $(q_{11})_n\rightarrow (q_{11})_{\infty}$. By the dominated convergence theorem,
\begin{gather}
c_n\rightarrow c_{\infty} = \frac{N-2l}{2}-\frac{1}{4}\int_{-\infty}^{\infty}\frac{1}{2\times 4^{1/N}}\det(g_0)^{-1/N}Ve^{u_0}(q_{11})_{\infty}e^{-2\alpha (q_{11})_{\infty}-\tilde{\psi}_{\infty}/N}.
\label{eq:limitingc}
\end{gather}
 Now we have two possibilities:
\begin{enumerate}
  \item $(q_{11})_{\infty}(0)=0$: The estimates in \ref{ineq:q11intermsof110} show that $(q_{11})_{\infty}=0$ identically. Hence $c_n\rightarrow \frac{N-2l}{2}$. Thus $V_n\rightarrow \frac{2N}{\tau}$.
  \item $(q_{11})_{\infty}(0)>0$: In this case, $(q_{11})_{\infty}$ is easily seen to be smooth. Equation \ref{eq:theeqforq22n} shows that passing to a further subsequence, $(q_{22})_n\rightarrow (q_{22})_{\infty}$ smoothly on compact sets. In fact, $(q_{22})_{\infty}= \frac{(q_{22})_{\infty}(0)}{(q_{11})_{\infty}(0)}(q_{11})_{\infty}$. Moreover, Equation \ref{eq:therealpartforc} shows that $x_n\rightarrow x_{\infty}$ smoothly on compact sets and $x_{\infty}=\frac{x_{\infty}(0)}{(q_{11})_{\infty}(0)}q_{11}$. Using \ref{ineq:comparingxnq11n}, \ref{eq:zerointegralofq12}, and the dominated convergence theorem, we see that $x_{\infty}=0$ identically. Since $\det(q)_{\infty}=0$ identically, this means that $(q_{11})_{\infty}(q_{22})_{\infty}=0$ identically. Thus $q_{22}=0$ identically. From Equation \ref{eq:uponintegrationforq} we see that $bq_{11}=0$, which implies that $b=0$. However, from \ref{eq:explicitaandc} we see that $b>0$ and hence we have a contradiction.
\end{enumerate}
Therefore, $(q_{11})_{\infty}(0)=0$ and hence $(q_{11})_{\infty}=0$ identically. The formula for $c$ (Equation \ref{eq:explicitaandc}) shows that for any sequence $\lambda_n\rightarrow \infty$, there exists a subsequence such that $V_n\rightarrow \frac{2N}{\tau}$. Thus $V_n\rightarrow \frac{2N}{\tau}$ and we are done.
\end{proof}
From Propositions \ref{prop:zerolimitlambda} and \ref{prop:infinitylimitlambda}, and the continuity of the volume in $\lambda$, we see using the intermediate value theorem that there exists a smooth solution for any admissible volume. \qed

\end{document}